\documentclass{amsart}
\usepackage[utf8]{inputenc}
\usepackage{fullpage}
\allowdisplaybreaks
\usepackage{amsmath,amssymb,amsthm}
\usepackage{mathtools}

\usepackage{verbatim}
\usepackage{enumerate}
\usepackage{hyperref}
\usepackage{dsfont}
\usepackage{color}
\usepackage{bbm}

\usepackage[margin=1in]{geometry} 
\usepackage[subnum]{cases}
\usepackage[noadjust]{cite}
\numberwithin{equation}{section}

\newtheorem{thm}{Theorem}[section]
\newtheorem{lemma}[thm]{Lemma}

\newtheorem{proposition}[thm]{Proposition}

\theoremstyle{definition}
\newtheorem{defin}[thm]{Definition}

\newcommand{\N}{\mathbb{N}}

\newcommand{\R}{\mathbb{R}}
\newcommand{\Z}{\mathbb{Z}}
\newcommand{\C}{\mathbb{C}}
\newcommand{\Cn}{\mathbb{C}^n}
\newcommand{\Sn}{\mathbb{S}^{2n-1}}
\newcommand{\la}{\lambda}

\newcommand{\sceil}[2]{\lceil #1 \rceil_{#2}}
\newcommand{\ind}{\mathbbm{1}}
\newcommand{\volSn}{\textnormal{vol}(\mathbb{S}^{2n-1})}

\newcommand{\boxb}{\square_b}

\newcommand{\boxN}{\square}

\newcommand{\hk}[2]{\mathcal{H}_{#1}(#2)}

\makeatletter
\newcommand{\skipitems}[1]{%
	\addtocounter{\@enumctr}{#1}%
}
\makeatletter

\title[A Tauberian Approach for the Kohn Laplacian]{A Tauberian Approach to Weyl's Law for the Kohn Laplacian on Spheres}

\author{Henry Bosch}
\address[Henry Bosch]{Harvard University, Department of Mathematics, Cambridge, MA 02138, USA}
\email{henrybosch@college.harvard.edu }

\author{Tyler Gonzales}
\address[Tyler Gonzales]{University of Wisconsin-Eau Claire, Department of Mathematics,
Eau Claire, WI 54701, USA}
\email{gonzaltj9215@uwec.edu}

\author{Kamryn Spinelli}
\address[Kamryn Spinelli]{Worcester Polytechnic Institute, Department of Mathematical Sciences, Worcester, MA 01609, USA}
\email{kpspinelli@wpi.edu}

\author{Gabe Udell}
\address[Gabe Udell]{Pomona College, Department of Mathematics,
Claremont, CA 91711}
\email{grua2017@mymail.pomona.edu}

\author{Yunus E. Zeytuncu}
\address[Yunus E. Zeytuncu]{University of Michigan--Dearborn, Department of Mathematics and Statistics,
Dearborn, MI 48128, USA}
\email{zeytuncu@umich.edu}

\thanks{This work is supported by NSF (DMS-1950102 and DMS-1659203). The work of the last author is also partially supported by a grant from the Simons Foundation (\#353525).}

\subjclass[2010]{Primary 32V05; Secondary 32V20}
\keywords{Kohn Laplacian, Weyl's Law, Karamata's Tauberian Theorem}

\begin{document}

\begin{abstract}
We compute the leading coefficient in the asymptotic expansion of the eigenvalue counting function for the Kohn Laplacian on the spheres. We express the coefficient as an infinite sum and as an integral.
\end{abstract}

\maketitle

\section{Introduction}

Let $\Sn$ denote the unit sphere in $\Cn$ where $n\geq 2$. The hypersurface $\Sn$ is an embedded CR manifold and the tangential Cauchy-Riemann operators $\overline{\partial}_b$ and $\overline{\partial}_b^*$ are defined on the corresponding Hilbert spaces. Furthermore, the Kohn Laplacian on $L^2(\Sn)$

$$\boxb=\overline{\partial}_b^*\overline{\partial}_b$$ 
is a linear self-adjoint densely defined closed operator. We refer to \cite{Boggess91CR} and \cite{CS01} for the detailed definitions. 

Our inspiration is the celebrated Weyl's law for Riemannian manifolds. In that setting, the eigenvalue counting function of the Laplace-Beltrami operator on a manifold $M$ has leading coefficient proportional to the volume of $M$ with a constant depending only on the dimension of $M$. There is also a corresponding result for the extension of this operator (also called the Hodge Laplacian, or the Laplace–de Rham operator) to differential forms. 

Motivated by the spectral theory for the Laplacian on Riemannian manifolds, one can investigate the spectrum of $\boxb$ on CR manifolds and its relation to the complex geometry of the underlying manifold. For example, in \cite{Fu2005} and \cite{Fu2008} Fu studied the spectrum of the $\overline{\partial}$-Neumann Laplacian $\boxN$ on smooth pseudoconvex domains $\Omega$ and related the distribution of the eigenvalues of $\boxN$ to the D'Angelo type of $b\Omega$. In \cite{Folland} Folland computed the spectrum of $\boxb$ on $\Sn$ on all differential form levels.  Recently, in \cite{REU17} and \cite{REU18} the authors used the spectrum of the Kohn Laplacian to prove the non-embeddability of the Rossi sphere.

For $\lambda>0$, let $N(\lambda)$ denote the number of positive eigenvalues (counting multiplicity) of $\boxb$ on $L^2(\Sn)$ that are less or equal to $\lambda$. It was noted in \cite{REU18} that $N(\lambda)$ grows on the order of $\lambda^n$ and later, in \cite{BanZey} (see the erratum), the leading coefficient in the asymptotic expansion was calculated as an infinite sum. In particular, \cite{BanZey} obtained
$$ \lim_{\lambda \to \infty} \frac{N(\lambda)}{\lambda^n} = \frac{1}{2^n n!}\sum_{q=1}^\infty \frac{1}{q^n} \bigg[\binom{n+q-2}{q}+{q-1 \choose n -2}\bigg] $$
by using careful counting arguments.

\subsection{Main Result}
In this paper, we obtain the same leading coefficient by another argument, namely by Karamata's Tauberian Theorem. We highlight that this technique is different than the one in \cite{BanZey}. Furthermore, we formulate the leading coefficient as a product of the volume of $\Sn$ and an integral that only depends on $n$. This representation resonates better with Weyl's law and is more amenable to generalization to other CR manifolds.

\begin{thm}\label{main}
Let $N(\lambda)$ be the eigenvalue counting function for $\boxb$ on $L^2(\Sn)$ as above. Then
\begin{align*}
\lim_{\lambda \to \infty} \frac{N(\lambda)}{\lambda^n} &= \frac{1}{2^n n!}\sum_{q=1}^\infty \frac{1}{q^n} \bigg[\binom{n+q-2}{q}+{q-1 \choose n -2}\bigg]\\
&=\volSn\frac{n-1}{n(2\pi)^n\Gamma(n+1)}\int_{-\infty}^{\infty}\left(\frac{t}{\sinh t}\right)^ne^{-(n-2)t}dt.
\end{align*}

\end{thm}

The integral representation is strikingly similar to (and indeed inspired by) a similar formula for the leading coeficient of the eigenvalue counting function for $\boxb$ acting on $(0,q)$-forms (where $q\geq 1$) in \cite{StanTar, Stanton}. However, we note that the expression in Theorem \ref{main} is not a special case of the expression in \cite{StanTar}. Stanton and Tartakoff's formula doesn't cover the case of functions, and some non-trivial adjustments are necessary to formulate the correct expression for the case $q=0$. We present a connection between our formula and the one in \cite{StanTar} in the last section. We list the task of finding the leading coefficient of the eigenvalue counting function of $\boxb$ acting on functions of a general pseudoconvex CR manifold of hypersurface type as an open problem.

\subsection{Ingredients} \label{sec:ingredients}
Before we present a proof of Theorem \ref{main} in the next section, we state some known facts about the eigenvalues of $\boxb$ on $\Sn$ and Karamata's Tauberian Theorem. 

The Kohn Laplacian acts on the space of $L^2$ differential forms on the sphere. In \cite{Folland}, Folland uses unitary representations to explicitly compute the eigenvalues and corresponding eigenspaces for the Kohn Laplacian on $(0,j)$-forms. In particular, he shows that eigenforms $$\bar{z_1}^{q-1} z_n^p \sum_{i=1}^{j+1} d\bar{z_1} \wedge \cdots \wedge \widehat{d\bar{z_i}} \wedge \cdots \wedge d\bar{z}_{j+1}$$ have corresponding eigenvalue $2 (q+j) (p+n-j-1)$, where the hat over a form indicates its exclusion from the wedge product.

We are interested in the case where $j=0$, corresponding to functions on the sphere. Folland (see also \cite{REU17, REU18}) explicitly shows that in this case
\[L^2(\Sn) = \bigoplus_{k=0}^\infty \hk{k}{\Sn} = \bigoplus_{p,q=0}^\infty \hk{p,q}{\Sn},\]
where $\hk{p,q}{\Sn}$ is the space of spherical harmonics of bidegree $p,q$.
Furthermore, each space $\hk{p,q}{\Sn}$ is an eigenspace of $\boxb$ with dimension
\[\textup{dim}(\mathcal{H}_{p,q}(\mathbb{S}^{2n-1}))=\binom{n+p-1}{p} \binom{n+q-1}{q}-\binom{n+p-2}{p-1}\binom{n+q-2}{q-1},\]
as computed by an inclusion-exclusion argument  (see \cite{Klima}). These spaces have corresponding eigenvalues $2q (p+n-1)$.

Karamata's Tauberian Theorem has been in used in \cite{Kac, StanTar} to understand the distribution of eigenvalues. We follow the statement in \cite[Theorem 1.1, page 57]{Arendt} where the reader can find further references.

\begin{thm}[Karamata]\label{karamata}
Let $\{\lambda_j\}_{j\in\mathbb{N}}$ be a sequence of positive real numbers such that $\sum_{j\in\mathbb{N}}e^{-\lambda_j t}$ converges for every $t>0$. Then for $n>0$ and $a\in\mathbb{R}$, the following are equivalent.
\begin{enumerate}
    \item $\lim_{t\rightarrow{0}^+}t^n\sum_{j\in\mathbb{N}}e^{-\lambda_jt}=a$
    \item $\lim_{\lambda\rightarrow{\infty}}\frac{N(\lambda)}{\lambda^n}=\frac{a}{\Gamma(n+1)}$
\end{enumerate}
where $N(\lambda)=\#\{\lambda_j:\lambda_j\leq\lambda\}$ is the counting function.
\end{thm}

The next section is dedicated to the proof of Theorem \ref{main}. First, we prove the expression for the leading coefficient as a series (recovering the result in \cite{BanZey}). Next, we express the leading coefficient as the volume of $\Sn$ times an integral.

\section{Tonelli Tactic To Tie Together Two Tauberian Terms}\label{mainsec}

Instead of considering $N(\lambda)$ directly we define the function $G(t)=\sum_{j\in \N} e^{-\lambda_j t}$
and consider $$\lim_{t\to 0^+} t^n\sum_{j\in \N} e^{-\lambda_j t}$$ where the sequence $\{\lambda_j\}_{j\in \N}$ is the sequence of all positive eigenvalues (with multiplicity) of $\boxb$ on the sphere. Once we compute this limit we can invoke Karamata's Theorem. Using the preliminaries from section \ref{sec:ingredients}, we have
\begin{align*}
G(t)&=\sum_{j\in \N} e^{-\lambda_j t}=
\sum_{q=1}^\infty \sum_{p=0}^\infty \dim(\mathcal{H}_{p,q}) e^{-2q(p+n-1) t}\\
&=\sum_{q=1}^\infty \sum_{p=0}^\infty \left(\binom{n+p-1}{p}\binom{n+q-1}{q}-\binom{n+p-2}{p-1}\binom{n+q-2}{q-1}\right) (e^{-2 t q})^{p+n-1}.
\end{align*}
Applying the standard recursive formula for binomial coefficients to the first product of binomial coefficients gives us\footnote{Note that we set $\binom{a}{b}=0$ if $a\leq 0$.} \begin{align*}
\binom{n+p-1}{p}\binom{n+q-1}{q} &= \binom{n+p-1}{p}\left[\binom{n+q-2}{q} + \binom{n+q-2}{q-1}\right] \\
&= \binom{n+p-1}{p} \binom{n+q-2}{q} + \binom{n+p-1}{p} \binom{n+q-2}{q-1} \\ 
&= \binom{n+p-1}{p} \binom{n+q-2}{q} + \left[\binom{n + p - 2}{p} + \binom{n + p - 2}{p - 1} \right] \binom{n+q-2}{q-1} \\
&= \binom{n+p-1}{p} \binom{n+q-2}{q} + \binom{n + p - 2}{p}\binom{n+q-2}{q-1} + \binom{n + p - 2}{p - 1} \binom{n+q-2}{q-1}.
\end{align*}
This allows us to rewrite $G(t)$ as a sum of two positive pieces by noticing that 
\[G(t) =  \sum_{q=1}^\infty \sum_{p=0}^\infty \left(\binom{n+p-1}{p} \binom{n+q-2}{q} + \binom{n + p - 2}{p}\binom{n+q-2}{q-1}\right) e^{-2 t q(p+n-1)}.\]
We label the parts as  
\[ G_1(t) =  \sum_{q=1}^\infty \sum_{p=0}^\infty \binom{n+p-1}{p} \binom{n+q-2}{q}  e^{-2 t q(p+n-1)}\]
and
\[G_2(t) =  \sum_{q=1}^\infty \sum_{p=0}^\infty  \binom{n + p - 2}{p}\binom{n+q-2}{q-1} e^{-2 t q(p+n-1)},\]
so that $G(t)= G_1(t)+ G_2(t)$. 

Our goal is to calculate $\lim_{t \to 0^+} t^n G(t)$, which we do by calculating $ \lim_{t \to 0^+} t^n G_1(t)$ and $\lim_{t \to 0^+} t^n G_2(t)$ separately. 

Theorem \ref{main} claims that $\lim_{\la\to \infty} \frac{N(\la)}{\la^n}$ can be written either as an infinite series or as an improper integral. The key to obtaining these two distinct forms is Tonelli's Theorem (see \cite[Theorem 8.8]{Rudin}). Since the summands in $G_1(t)$ and $G_2(t)$ are positive, we can exchange the order of the infinite sums in each. When the outer sum is over $q$, we can show that $\lim_{t\to 0^+} G_1(t)$ is an infinite series and $\lim_{t\to 0^+} G_2(t)$ is an improper integral; exchanging the order of summation allows us to express $\lim_{t\to 0^+} G_1(t)$ as an integral and $\lim_{t\to 0^+} G_2(t)$ as an infinite series.

\subsection{Serious Series for Spherical Spectra}\label{series}

In this part we prove the infinite series formula for $\lim_{\la\to \infty} \frac{N(\la)}{\la^n}$. The computations for $G_1(t)$ and $G_2(t)$ are quite similar. In each instance we will apply the Dominated Convergence Theorem by leveraging Lemma \ref{doubleDCT}.

In each computation we eventually exchange the limit as $t\to 0^+$ with a sum and in each case we require the following limit calculation, which follows from L'Hopital's rule.

\begin{lemma}\label{limit}
Let $\alpha > 0$. Then
\[\lim_{t\to 0} \frac{t^n}{(1-e^{-\alpha t})^n}=\alpha^{-n}.\]
\end{lemma}

In order to exchange limits with sums we will apply the Dominated Convergence Theorem. This next technical lemma will come in handy for showing the conditions of the Dominated Convergence Theorem are satisfied in each case.  

\begin{lemma}\label{doubleDCT}
For $n\in\mathbb{N}$, there exists $M>0$ such that 
\[f(x)=\frac{x^n e^{2x(n-1)}}{(e^{2x}-1)^n}<M\] for all $x\in (0,\infty)$.
\end{lemma}
\begin{proof}
The function $f$ is continuous on $(0,\infty)$. If we show that $\lim_{x\to \infty} f(x)<\infty$ and $\lim_{x\to 0^+} f(x)<\infty$ it will follow that there exists some $M$ such that $f(x)<M$ on $(0,\infty)$. Indeed, $\lim_{x\to \infty} f(x) = 0$. Further, L'Hopital's rule gives $\lim_{x\to 0^+}\frac{x}{e^{2x}-1}=\frac{1}{2}$
and so $\lim_{x\to 0^+}\frac{x^n e^{2x(n-1)}}{(e^{2x}-1)^n} =\frac{1}{2^n}$.
\end{proof}

Now we are ready to compute $\lim_{t\to 0^+}t^n G_1(t)$.

\begin{proposition}
\label{firstsum}
\[\lim_{t\to 0^+} t^n G_1(t) = \frac{1}{2^n}\sum_{q=1}^\infty \binom{n+q-2}{q}\frac{1}{q^n}.\]
\end{proposition}
\begin{proof}

We start by coming up with an expression for $G_1(t)$ containing only a single sum.
For $|z| < 1$, there is the Taylor series expansion
$$\frac{1}{(1-z)^n} = \sum_{p=0}^\infty \binom{n+p-1}{p}z^p = \sum_{p=1}^\infty \binom{n+p-2}{p-1}z^{p-1}.$$ 
Let $z = e^{-2tq}$; then
\begin{align*}
    G_1(t) &=  \sum_{q=1}^\infty \sum_{p=0}^\infty \binom{n+p-1}{p} \binom{n+q-2}{q}  e^{-2 t q(p+n-1)}\\
    &=
    \sum_{q=1}^\infty \left(\binom{n+q-2}{q}\sum_{p=0}^\infty \binom{n+p-1}{p} (e^{-2 t q})^{p+n-1}\right) \\
    &=\sum_{q=1}^\infty \binom{n+q-2}{q}\sum_{p=0}^\infty \binom{n+p-1}{p} z^{p+n-1}\\
    &=\sum_{q=1}^\infty \binom{n+q-2}{q}z^{n-1}\sum_{p=0}^\infty \binom{n+p-1}{p} z^{p} \\
    &=\sum_{q=1}^\infty \binom{n+q-2}{q}\frac{z^{n-1}}{(1-z)^n}\\
    &=\sum_{q=1}^\infty \binom{n+q-2}{q}\frac{(e^{-2 t q})^{n-1}}{(1-e^{-2 t q})^n}\\
    &=\sum_{q=1}^\infty \binom{n+q-2}{q}\frac{e^{2 t q}}{(e^{2 t q}-1)^n}.
\end{align*}

For positive $t$, $\frac{t^n(e^{2 t q})}{(e^{2 t q}-1)^n}\leq \frac{1}{q^n}\frac{(qt)^n(e^{2 t q (n-1)})}{(e^{2 t q}-1)^n} =\frac{f(qt)}{q^n}\leq \frac{M}{q^n}$ where $f$ is as in Lemma \ref{doubleDCT}.
Hence, $t^n \binom{n+q-2}{q} \frac{(e^{2 t q})}{(e^{2 t q}-1)^n}\leq M \binom{n+q-2}{q} \frac{1}{q^n}$ for all $t$ and since $\sum_{q=1}^\infty M \binom{n+q-2}{q} \frac{1}{q^n}$ converges, we are free to apply the Dominated Convergence Theorem to $\lim_{t\to 0^+}t^n G_1(t)$.

This technical justification allows us to exchange the order of the limit and summation in $\lim_{t\to 0^+}t^n G_1(t)$ and conclude that
$$\lim_{t\to 0^+}\sum_{q=1}^\infty \binom{n+q-2}{q}\frac{t^n(e^{-2 t q})^{n-1}}{(1-e^{-2 t q})^n} = \sum_{q=1}^\infty \lim_{t\to 0^+} \binom{n+q-2}{q}\frac{t^n(e^{-2 t q})^{n-1}}{(1-e^{-2 t q})^n} =\sum_{q=1}^\infty \binom{n+q-2}{q}\frac{1}{(2q)^n},$$
where the final equality follows from Lemma \ref{limit}.

\end{proof}

Next we move to the second piece and compute $\lim_{t\to 0^+} t^n G_2(t)$. 

\begin{proposition}\label{secondsum}
\[\lim_{t\to 0^+} t^n G_2(t) = \frac{1}{2^n}\sum_{q=1}^\infty \binom{q-1}{n-2}\frac{1}{q^n}.\]
\end{proposition}
\begin{proof}
We start out by manipulating the form of $G_2(t)$ as we did with $G_1(t)$ but this time we apply Tonelli's Theorem to switch the order of summation. In our calculation we will make use of the substitutions $z=e^{-2t(p+n-1)}$ and $w=p+n-1$ and we will apply the power series expansion of $\frac{1}{(1-z)^{n-1}}$.

\begin{align*}
    G_2(t)&=\sum_{q=1}^\infty \sum_{p=0}^\infty \binom{n+p-2}{p}\binom{n+q-2}{q-1}e^{-2tq(p+n-1)}
    =\sum_{p=0}^\infty \binom{n+p-2}{p} \sum_{q=1}^\infty \binom{n+q-2}{q-1}e^{-2tq(p+n-1)}\\
    &=\sum_{p=0}^\infty \binom{n+p-2}{p} z \sum_{q=1}^\infty \binom{n+q-2}{q-1}z^{q-1}=\sum_{p=0}^\infty \binom{n+p-2}{p} \frac{z}{(1-z)^n}\\
    &=\sum_{p=0}^\infty \binom{n+p-2}{n-2} \frac{e^{-2t(p+n-1)}}{(1-e^{-2t(p+n-1)})^n}=\sum_{w=n-1}^\infty \binom{w-1}{n-2}\frac{e^{-2tw}}{(1-e^{-2tw})^n}\\
    &=\sum_{w=1}^\infty \binom{w-1}{n-2}\frac{e^{-2tw}}{(\frac{e^{2tw}-1}{e^{2tw}})^n}=\sum_{w=1}^\infty \binom{w-1}{n-2} \frac{e^{2tw(n-1)}}{(e^{2tw}-1)^n}.
\end{align*}
By Lemma \ref{doubleDCT}, $t^n \binom{w-1}{n-2} \frac{e^{2tw(n-1)}}{(e^{2tw}-1)^n}=\binom{w-1}{n-2}\frac{f(tw)}{w^n} \leq M\binom{w-1}{n-2}\frac{1}{w^n}.$ As \(\sum_{w=0}^\infty M\binom{w-1}{n-2}\frac{1}{w^n}\) converges, we apply the Dominated Convergence Theorem to $\lim_{t\to 0^+} t^nG_2(t)$ in order to exchange the limit and the sum. Thus,
$$\lim_{t\to 0^+} t^n G_2(t) = \sum_{w=1}^\infty \lim_{t\to 0^+}\binom{w-1}{n-2}\frac{t^n e^{-2tw}}{(\frac{e^{2tw}-1}{e^{2tw}})^n}=\frac{1}{2^n}\sum_{w=1}^\infty \binom{w-1}{n-2}\frac{1}{w^n}.   $$
\end{proof}

Applying the Tauberian Theorem \ref{karamata} in combination with Propositions \ref{firstsum} and \ref{secondsum} proves that
\[\lim_{\la\to \infty} \frac{N(\la)}{\la^n} =  \frac{1}{2^n n!}\sum_{q=1}^\infty \frac{1}{q^n}\left[\binom{n+q-2}{q}+\binom{q-1}{n-2}  \right].\]
This concludes the proof of the first identity in Theorem \ref{main}.

\subsection{Isn't it Interesting: Intense Integral is Identical In Immensity}\label{integral}

In this subsection we will prove the second part of Theorem \ref{main}.  We start off by re-analyzing $G_1(t)$:
\begin{align*}
    G_1(t) &= \sum_{q=1}^\infty \sum_{p=0}^\infty \binom{n+p-1}{p} \binom{n+q-2}{q}  e^{-2 t q(p+n-1)} \\
    &= \sum_{p=0}^\infty \sum_{q=1}^\infty \binom{n+p-1}{p} \binom{n+q-2}{q}  e^{-2 t q(p+n-1)} \\
    &= \sum_{p=0}^{\infty}\binom{n+p-1}{p}\sum_{q=1}^{\infty}\binom{n+q-2}{q}  (e^{-2 t (p + n - 1)})^{q} \\
    &= \sum_{p=0}^{\infty}\binom{n+p-1}{p}\left( \frac{1}{(1 - e^{-2t(p + n - 1)})^{n-1}} - 1 \right) \\
    &= \sum_{w=n-1}^{\infty}\binom{w}{n-1}\left( \frac{1}{(1 - e^{-2tw})^{n-1}} - 1 \right) \\
\end{align*}
where the switching of the order of summation is justified by Tonelli's Theorem and we have taken $w = n + p - 1$. 
To calculate the limit as $t\to 0^+$, we need the following lemma.
\begin{lemma}\label{gross1}
    Fix an integer $r \geq 1$ and consider the function
    \[f_{r}(x) = x^r \left( \frac{1}{(1 - e^{-2x})^r} - 1 \right)\]
    defined for $x > 0$. Then
    \begin{enumerate}
        \item $f_{r}(x) > 0$.
        \item $f_{r}'(x) < 0$ for sufficiently large $x$.
        \item $f_{r}$ is bounded and $\int_{0}^\infty f_{r}(x) dx < \infty.$
    \end{enumerate}
\end{lemma}
\begin{proof}

We can tell $f_{r}(x)>0$ since $e^{-2x}<1$ so $0<1-e^{-2x}<1$ and hence $\frac{1}{(1-e^{-2x})^r}>1$, proving (1). 

To establish (2), a derivative calculation shows that 
\begin{align*}
    f_r'(x) &= rx^{r-1}\left(\frac{1}{(1 - e^{-2x})^r} - 1\right) -  x^r\frac{2re^{-2x}}{(1 - e^{-2x})^{r+1}} \\
    &= \frac{r x^{r-1}}{(1 - e^{-2x})^{r+1}}\left[1 - e^{-2x} - (1 - e^{-2x})^{r+1} - 2xe^{-2x} \right] \\
    &= \frac{r x^{r-1}}{(1 - e^{-2x})^{r+1}e^{2x}}\left[e^{2x}- 1 - (1 - e^{-2x})^{r+1}e^{2x} - 2x \right].
\end{align*}

Define $l(x) = e^{2x} - 1 - (1 - e^{-2x})^{r+1}e^{2x} - 2x.$ To prove (2), it suffices to show that $l(x) < 0$ for sufficiently large $x$. For this, let $y = e^{-2x}$ and note that 
$$\lim_{x \to \infty} e^{2x} - 1 - (1 - e^{-2x})^{r+1}e^{2x} = \lim_{y \to 0^+} \frac{1 - y - (1-y)^{r+1}}{y} = \lim_{y \to 0^+} -1 + (r+1)(1-y)^{r} = r.$$ Looking at the definition of $l$, 
the first three terms converge to $r$ and the last goes to $-\infty$, so $l(x) < 0$ for sufficiently large $x$.

To prove (3) we will show that $f_r(x)$ is bounded by a constant near 0 and bounded by an exponentially decaying function for large $x$.
The function $f_{r}(x)$ isn't defined at $x=0$, but Lemma \ref{limit} implies that $\lim_{x\to 0} f_{r}(x)<\infty$ so the discontinuity at 0 is removable. This implies that $f_r(x)$ is bounded on $[0,\frac{\log(2)}{2}]$. To bound $f_r(x)$ on $[\frac{\log 2}{2},\infty)$ we note
\[\frac{1}{(1-e^{-2x})^r}-1=\frac{1}{((e^{2x}-1)/e^{2x})^r}-1= \frac{(e^{2x})^r-(e^{2x}-1)^r}{(e^{2x}-1)^r}.
\]
Using the formula for a difference of $r$th powers we obtain the following expression for the numerator of the above fraction
\begin{align*}
    (e^{2x})^r-(e^{2x}-1)^r&=(e^{2x}-(e^{2x}-1)) \sum_{i=0}^{r-1} (e^{2x})^i (e^{2x}-1 )^{r-1-i}\\
    &\leq \sum_{i=0}^{r-1} (e^{2x})^i (e^{2x})^{r-1-i}=re^{2x(r-1)}.
\end{align*}
Hence, \[\frac{1}{(1-e^{-2x})^r}-1 \leq \frac{re^{2x(r-1)}}{(e^{2x}-1)^r}.\]
Note $e^{2x}-1 \geq \frac{e^{2x}}{2}$ for $x \geq \frac{\log(2)}{2}$ so 
\[\frac{1}{(1-e^{-2x})^r}-1 \leq \frac{re^{2x(r-1)}}{(e^{2x}/2)^r}=\frac{r2^r}{e^{2x}}.\]
Applying this to $f_r(x)$ gives that when $x \geq \frac{\log(2)}{2},$
\[f_r(x)\leq \frac{r2^rx^r}{e^{2x}}.\]
Thus $f_r(x)$ is dominated by an exponentially decaying function for sufficiently large $x$, so $\int_0^\infty f_r(x)\,dx<\infty.$

\end{proof}

It will be helpful to make the following definition: 
\begin{defin}
For real $\alpha \neq 0$, define the scaled ceiling function $\lceil \cdot \rceil_{\alpha}: \R \to \R$ by
\[\lceil x \rceil_{\alpha} = \alpha\lceil x/\alpha \rceil.\]
\end{defin} 
For example, $\lceil 7 \rceil_{3} = 3 \lceil \frac{7}{3} \rceil = 3 \cdot 3 = 9$ and $\lceil 6 \rceil_{2} = 2 \lceil \frac{6}{2} \rceil = 2 \cdot 3 = 6$. Then we have the following:
\begin{proposition}
$\sceil{x}{\alpha}$ is $x$ rounded up to the nearest integral multiple of $\alpha$, i.e. \[\lceil x \rceil_{\alpha} = \min\{n\alpha: n \in \Z, n\alpha \geq x\}.\] 
\end{proposition}
\begin{proof}
\[\sceil{x}{\alpha} = \alpha \min\{n \in \Z: n \geq x/\alpha\} = \alpha \min\{n \in \Z: \alpha n \geq x\} \]
and the result follows.
\end{proof}

The scaled ceiling function has the following properties:
\begin{enumerate}
    \item Fix $x \in \R$, $\alpha > 0$. Then $0 \leq \sceil{x}{\alpha} - x < \alpha$.
    \item Fix $x \in \R$. Then $\lim_{\alpha \to 0^+} \sceil{x}{\alpha} = x$.
    \item Let $f: [a,b] \to \R$ be monotonically decreasing. Fix $0 < \alpha \leq b - a$. Then for all $x \in [a,b-\alpha]$, $f(\sceil{x}{\alpha}) \leq f(x)$.
\end{enumerate}

The next few propositions will allow us to compute $\lim_{t\to 0^+} G_1(t)$.

\begin{proposition} \label{leading-term}
\[
\lim_{t \to 0^+} t^n \sum_{w=n-1}^{\infty}w^{n-1} \left( \frac{1}{(1 - e^{-2tw})^{n-1}} - 1 \right) = \int_{0}^\infty x^{n-1} \left( \frac{1}{(1 - e^{-2x})^{n-1}} - 1 \right)dx. \] 
\end{proposition}
\begin{proof}
Manipulating the sum, we have
\begin{align*}
    t^n \sum_{w=n-1}^{\infty} w^{n-1} \left( \frac{1}{(1 - e^{-2tw})^{n-1}} - 1 \right) 
    &= \sum_{w=n-1}^{\infty}t^{n} w^{n-1} \left( \frac{1}{(1 - e^{-2tw})^{n-1}} - 1 \right)  \\
    &= \sum_{w=n-1}^{\infty} \int_{w-1}^w t^{n} \lceil w' \rceil ^{n-1} \left( \frac{1}{(1 - e^{-2t\lceil w' \rceil})^{n-1}} - 1 \right) dw' \\
    &= \int_{n-2}^\infty t^{n} \lceil w' \rceil ^{n-1} \left( \frac{1}{(1 - e^{-2t\lceil w' \rceil})^{n-1}} - 1 \right) dw' \\
    &= \int_{t(n-2)}^\infty (t\lceil x/t \rceil) ^{n-1} \left( \frac{1}{(1 - e^{-2t\lceil x/t \rceil})^{n-1}} - 1 \right) dx \\
    &= \int_{t(n-2)}^\infty \sceil{x}{t} ^{n-1} \left( \frac{1}{(1 - e^{-2 \sceil{x}{t}})^{n-1}} - 1 \right) dx\\
    &=\int_{t(n-2)}^\infty f_{n-1}(\sceil{x}{t}) dx.
\end{align*}
Fix $C$ such that  $f_{n-1}'(x) < 0$ for $x \geq C$ and fix $M$ such that $f_{n-1}(x)<M$. Then
\[f_{n-1}(\sceil{x}{t}) \leq M \ind_{\{x < C\}} + f_{n-1}(x) \]
for all $x > 0$. Since the integral of the right hand side is finite, we may apply dominated convergence to see that
\[\lim_{t \to 0^+} \sum_{w=n-1}^{\infty}t^n w^{n-1} \left( \frac{1}{(1 - e^{-2tw})^{n-1}} - 1 \right)dx = \int_{0}^\infty x^{n-1} \left( \frac{1}{(1 - e^{-2x})^{n-1}} - 1 \right)dx \]
which completes the proof.
\end{proof}

The following proposition will help us compute the rest of the limits we need before we can tackle $\lim_{t\to 0^+}t^nG(t)$ in its entirety.
\begin{proposition}\label{goesto0}
Suppose that $a_t(w)$ is a positive function of real $t$ and integer $w\geq 1$ such that $\lim_{t\to 0^+}a_t(w)=0$ for each $w$ and $\lim_{t\to 0^+}\sum_{w=1}^\infty a_t(w) =M<\infty$. Then \[\lim_{t\to 0^+} \sum_{w=1}^\infty \frac{a_t(w)}{w} = 0.\] 
\end{proposition}
\begin{proof}
Let $\epsilon>0$ be arbitrary and let $k>2\frac{M}{\epsilon}-1$ be an integer. Since $\lim_{t\to 0^+} a_t(w)=0$ for each $w$, there exists some $T$ such that for all $0<t<T$ and all $w\leq k$, $a_t(w)\leq \frac{\epsilon}{2k}$. Thus for $t<T$, we have
\begin{align*}
    \sum_{w=1}^\infty \frac{a_t(w)}{w}&=\sum_{w=1}^k \frac{a_t(w)}{w}+\sum_{w=k+1}^\infty \frac{a_t(w)}{w}\\
    &\leq\sum_{w=1}^k a_t(w) +\frac{1}{k+1}\sum_{w=k+1}^\infty a_t(w)\leq \sum_{w=1}^k a_t(w) +\frac{1}{k+1}\sum_{w=1}^\infty a_t(w)\\
    &\leq \sum_{w=1}^k \frac{\epsilon}{2k} +\frac{M}{k+1}\leq \frac{\epsilon}{2}+\frac{\epsilon}{2}=\epsilon.
\end{align*}
Therefore $\lim_{t\to 0^+} \sum_{w=1}^\infty \frac{a_t(w)}{w} = 0.$
\end{proof}

\begin{proposition}
\label{vanishing-terms}
Fix $0 \leq k < n -1$. Then \[
\lim_{t \to 0^+} t^n \sum_{w=n-1}^{\infty}w^k \left( \frac{1}{(1 - e^{-2tw})^{n-1}} - 1 \right) = 0.\]
\end{proposition}
\begin{proof}
The observation that $0\leq  \sum_{w=n-1}^{\infty}t^n w^k \left( \frac{1}{(1 - e^{-2tw})^{n-1}} - 1 \right)\leq   \sum_{w=n-1}^{\infty}\frac{1}{w} t^n w^{n-1} \left( \frac{1}{(1 - e^{-2tw})^{n-1}} - 1 \right)$ allows us to apply Propositions \ref{goesto0} and \ref{leading-term} and arrive at the desired conclusion.
\end{proof}

With Propositions \ref{leading-term} and \ref{vanishing-terms} in hand, we can move forward. We break $\binom{w}{n-1}$ into a polynomial in $w$, as follows:
\[\binom{w}{n-1} = \frac{w!}{(n-1)! (w-n+1)!} = \frac{w(w-1)\cdots(w-n+2)}{(n-1)!} = \frac{w^{n-1}}{(n-1)!} + \sum_{k=0}^{n-2} a_k w^k.\] 
We write
\begin{align}
\lim_{t \to 0^+} t^n G_1(t)  &=  \lim_{t \to 0^+} t^n \sum_{w=n-1}^{\infty} \binom{w}{n-1}\left( \frac{1}{(1 - e^{-2tw})^{n-1}} - 1 \right) \nonumber \\
&= \lim_{t \to 0^+} \left[ \frac{t^n}{(n-1)!}\sum_{w=n-1}^{\infty}w^{n-1}\left( \frac{1}{(1 - e^{-2tw})^{n-1}} - 1 \right) + t^n \sum_{k=0}^{n-2}a_k  \sum_{w=n-1}^{\infty}w^{k}\left( \frac{1}{(1 - e^{-2tw})^{n-1}} - 1 \right)\right] \nonumber \\
&= \frac{1}{(n-1)!}\int_{0}^\infty x^{n-1} \left( \frac{1}{(1 - e^{-2x})^{n-1}} - 1 \right)dx. \label{eq:int-G1}
\end{align}
The final equality was obtained by applying Propositions \ref{leading-term} and \ref{vanishing-terms}.

Now we move on to $G_2(t)$. We have
\begin{align*}
    G_2(t) &=  \sum_{q=1}^\infty \sum_{p=0}^\infty  \binom{n + p - 2}{p}\binom{n+q-2}{q-1} e^{-2 t q(p+n-1)} \\
    &= \sum_{q=1}^\infty\binom{n+q-2}{q-1} e^{-2tq(n-1)}\sum_{p=0}^\infty  \binom{n + p - 2}{p} (e^{-2 t q})^p \\
    &= \sum_{q=1}^\infty\binom{n+q-2}{q-1} \frac{e^{-2tq(n-1)}}{(1 - e^{-2tq})^{n-1}} \\
    &= \sum_{q=1}^\infty\binom{n+q-2}{n-1} \frac{e^{-2tq(n-1)}}{(1 - e^{-2tq})^{n-1}}. 
\end{align*}
This next lemma plays a role in our analysis of $G_2(t)$ which is analogous to the role of Lemma \ref{gross1} in our analysis of $G_1(t)$.
\begin{lemma}
Fix integers $r \geq 0$ and $s \geq 1$, and consider the function
\[g_{r,s}(x) = x^r \frac{e^{-2xs}}{(1-e^{-2x})^{s}} \] defined for $x > 0$.
Then
\begin{enumerate}
    \item $g_{r,s}(x) > 0.$
    \item $g'_{r,r} (x)< 0$ .
    \item If $r \geq s$, $g_{r,s}(x)$ is bounded and $\int_{0}^\infty g_{r,s}(x) dx < \infty$.
    \item If $r \geq s + 1$,
    $\lim_{x \to 0} g_{r,s}(x) = 0$.
\end{enumerate}
\end{lemma}
\begin{proof}
Suppose $r\geq0$ and $s\geq1$ are integers and that $x>0$. As $x^r>0$, $e^{-2xs}>0$, and $1 - e^{-2x} > 0$, we have $g_{r,s}(x) > 0$ which shows (1). To prove (2) notice that we may write
$$g_{r,r}(x)=x^r \frac{e^{-2xr}}{(1-e^{-2x})^{r}}=\frac{x^r}{(e^{2x}-1)^r}.$$
Since the function $x/(e^{2x}-1)$ is decreasing, it follows that $g'_{r,r}(x)<0$.

For part (3), fix $M \in \R$ such that for all $x \geq M$, $e^{2x} - 1 \geq \frac{e^{2x}}{2}$ and $(e^{2x})^{1/2} \geq x^r$. Then for $x \geq M$,
\[g_{r,s}(x) = \frac{x^r}{(e^{2x} - 1)^s} \leq \frac{x^r}{\left( \frac{e^{2x}}{2} \right)^s} = 2^s \frac{x^r}{(e^{2x})^s} \leq 2^s \frac{(e^{2x})^{1/2}}{(e^{2x})^s} = \frac{2^s}{(e^{2x})^{s-1/2}}.\]
As $s \geq 1$, the right-hand side has finite integral over $[M,\infty)$ and hence $\int_M^\infty g_{r,s}(x)\ dx < \infty$. The integral of $g_{r,s}(x)$ over $[0,M]$ is also finite because we can extend $g_{r,s}(x)$ to a continuous, bounded function on this compact interval. Adding these two parts together shows that $$\int_0^\infty g_{r,s}(x)\ dx < \infty.$$

It remains to show $\lim_{x\rightarrow 0}g_{r,s}(x)=0$ whenever $r\geq s+1$. First notice that Lemma \ref{limit} can be used to show that
$$\lim_{x\rightarrow 0}\frac{x^r}{(1-e^{-2x})^s}=\lim_{x\rightarrow 0}x^{r-s}\lim_{x\rightarrow 0}\frac{x^s}{(1-e^{-2x})^s}=0\cdot2^{-n}=0.$$
Therefore, when $r\geq s+1$, we have
$$\lim_{x\rightarrow 0}x^r\frac{e^{-2sx}}{(1-e^{-2x})^s}=\lim_{x\rightarrow 0}e^{-2xs}\lim_{x\rightarrow 0}\frac{x^r}{(1-e^{-2x})^s}=1\cdot0=0$$
as we needed to show, this proves (4).
\end{proof}

Now we move on to propositions which, similarly to our strategy for analyzing $G_1(t)$, allow us to break up the binomial coefficient in $G_2(t)$ into a polynomial and analyze the resulting sums separately.
\begin{proposition}
Fix $0 \leq k < n -1$. Then \[
\lim_{t \to 0^+} t^n \sum_{q=1}^{\infty}q^k \frac{e^{-2tq(n-1)}}{(1 - e^{-2tq})^{n-1}} = 0.\] 
\end{proposition}
\begin{proof}
We recognize that this expression can be rewritten in terms of $g_{n,n-1}(tq)$: 
\[ t^n\sum_{q=1}^{\infty}q^k \frac{e^{-2tq(n-1)}}{(1 - e^{-2tq})^n} = \sum_{q=1}^{\infty}\frac{1}{q^{n-k}}(tq)^n \frac{e^{-2tq(n-1)}}{(1 - e^{-2tq})^{n-1}} = \sum_{q=1}^\infty \frac{1}{q^{n-k}} g_{n,n-1}(tq).\]
Since $g_{n,n-1}$ is bounded and $n-k \geq 2$, the summand is dominated by some constant times  $\frac{1}{q^{2}}$. Therefore we can apply dominated convergence to conclude that the limit is 0, since the summand converges pointwise to 0 by the previous lemma.
\end{proof}
\begin{proposition}\label{Gt} 
\[\lim_{t \to 0^+} t^n \sum_{q=1}^{\infty}q^{n-1} \frac{e^{-2tq(n-1)}}{(1 - e^{-2tq})^{n-1}} = \int_{0}^\infty x^{n-1} \frac{e^{-2x(n-1)}}{(1 - e^{-2x})^{n-1}} dx.\]
\end{proposition}
\begin{proof}
Write
\begin{align*}
    t^n \sum_{q=1}^{\infty}q^{n-1} \frac{e^{-2tq(n-1)}}{(1 - e^{-2tq})^{n-1}}
    &=  \sum_{q=1}^{\infty}\int_{q-1}^q t^n \lceil q' \rceil^{n-1} \frac{e^{-2t\lceil q' \rceil(n-1)}}{(1 - e^{-2t \lceil q' \rceil})^{n-1}} dq' \\
    &= \int_{0}^\infty t^n \lceil q' \rceil^{n-1} \frac{e^{-2t\lceil q' \rceil(n-1)}}{(1 - e^{-2t \lceil q' \rceil})^{n-1}} dq' \\
    &= \int_{0}^\infty (t \lceil x/t \rceil)^{n-1} \frac{e^{-2t\lceil x/t \rceil(n-1)}}{(1 - e^{-2t \lceil x/t \rceil})^{n-1}} dx \\
    &= \int_{0}^\infty \sceil{x}{t}^{n-1} \frac{e^{-2\sceil{x}{t}(n-1)}}{(1 - e^{-2\sceil{x}{t}})^{n-1}} dx.
\end{align*}
The integrand is exactly $g_{n-1,n-1}(\sceil{x}{t})$, and is thus dominated by $g_{n-1,n-1}$
since $g_{n-1,n-1}$ is a decreasing function. Therefore we can apply dominated convergence to find that
\[\lim_{t \to 0^+}  t^n \sum_{q=1}^{\infty}q^{n-1} \frac{e^{-2tq(n-1)}}{(1 - e^{-2tq})^{n-1}} = \int_{0}^\infty x^{n-1} \frac{e^{-2x(n-1)}}{(1 - e^{-2x})^{n-1}} dx.\]
\end{proof}
We can expand $\binom{n+q-2}{n-1}$ as \[\binom{n+q-2}{n-1} = \frac{q^{n-1}}{(n-1)!} + \sum_{k=0}^{n-2} b_k q^k.\]
Therefore,
\begin{align*}
\lim_{t \to 0^+} t^n G_2(t) 
&= \lim_{t \to 0^+} t^n \sum_{q=1}^\infty \binom{n+q-2}{n-1} \frac{e^{-2tq(n-1)}}{(1 - e^{-2tq})^{n-1}} \\
&= \lim_{t \to 0^+} \left[ \frac{1}{(n-1)!} t^n \sum_{q=1}^\infty q^{n-1} \frac{e^{-2tq(n-1)}}{(1 - e^{-2tq})^{n-1}} + t^n \sum_{k=0}^{n-2} b_k \sum_{q=1}^\infty q^k \frac{e^{-2tq(n-1)}}{(1 - e^{-2tq})^{n-1}}  \right] \\
&= \frac{1}{(n-1)!} \int_{0}^\infty x^{n-1} \frac{e^{-2x(n-1)}}{(1 - e^{-2x})^{n-1}} dx \\
&= \frac{1}{(n-1)!} \int_{0}^\infty x^{n-1} \frac{1}{(e^{2x} - 1)^{n-1}} dx.
\end{align*}
Next, we put the two parts of $G$ back together in the limit. This gives us
\begin{align*}
    \lim_{t \to 0^+} t^n G(t) &= \lim_{t \to 0^+} t^n G_1(t) + \lim_{t \to 0^+} t^n G_2(t) \\
    &= \frac{1}{(n-1)!} \left[\int_{0}^\infty x^{n-1} \left( \frac{1}{(1 - e^{-2x})^{n-1}} - 1 \right)dx + \int_{0}^\infty x^{n-1} \frac{1}{(e^{2x} - 1)^{n-1}} dx \right] \\
    &= \frac{1}{(n-1)!}\int_{0}^\infty x^{n-1} \left( \frac{1}{(1 - e^{-2x})^{n-1}} - 1 + \frac{1}{(e^{2x} - 1)^{n-1}} \right) dx. 
\end{align*}
We now have an expression for $\lim_{t\to 0^+} G(t)$ in terms of an integral so we could have chosen to stop here. Instead, we will press on and apply a few more tricks in order to arrive at 
the expression in Theorem \ref{main}. One reason we prefer the expression in Theorem \ref{main} is its similarity to a closely related result in \cite{StanTar}.

To continue on our path of manipulating the combined integral, we apply integration by parts with 
\begin{align*}
    u &= \frac{1}{(1 - e^{-2x})^{n-1}} - 1 + \frac{1}{(e^{2x} - 1)^{n-1}},\\
    dv&=x^{n-1}dx,\\
    du &= -(n-1)\left(\frac{2e^{-2x}}{(1-e^{-2x})^n} + \frac{2e^{2x}}{(e^{2x} - 1)^n}\right) dx = -2(n-1)\frac{e^{(n-2)x} + e^{-(n-2)x}}{(e^x - e^{-x})^n}dx = -\frac{n-1}{2^{n-2}} \frac{\cosh((n-2)x)}{\sinh(x)^n} dx,
\end{align*}
and $v=n^{-1}x^n$.
This gives us
\begin{align*}
    \lim_{t \to 0^+} t^n G(t) &= 
    \frac{1}{n!} x^n \left(\frac{1}{(1 - e^{-2x})^{n-1}} - 1 + \frac{1}{(e^{2x} - 1)^{n-1}}   \right)\bigg|_{0}^\infty + \frac{n-1}{2^{n-2}n!}\int_0^\infty x^n \frac{\cosh((n-2)x)}{\sinh(x)^n} dx \\
    &= \frac{n-1}{2^{n-2}n!}\int_0^\infty \frac{x^n \cosh((n-2)x)}{\sinh(x)^n} dx.
\end{align*}
The boundary term clearly vanishes at $\infty$ because $x^n \left( \frac{1}{(1 - e^{-2x})^{n-1}} - 1 \right)$ and $\frac{x^n}{(e^{2x} - 1)^{n-1}}$ each vanish at $\infty$. To see that it vanishes at 0, apply L'Hopital's rule to the quotients $\frac{x}{1-e^{-2x}}$ and $\frac{x}{e^{2x}-1}$.
Since the integrand is even, we further have
\begin{align*}
     \lim_{t \to 0^+} t^n G(t)
    &= \frac{n-1}{2^{n-2}n!}\int_0^\infty \frac{x^n \cosh((n-2)x)}{\sinh(x)^n} dx \\
    &= \frac{n-1}{2^{n-1}n!}\int_{-\infty}^\infty  \frac{x^n \cosh((n-2)x)}{\sinh(x)^n} dx \\
    &= \frac{n-1}{2^{n-2}n!} \left( \int_{-\infty}^\infty  \frac{x^n}{\sinh(x)^n} e^{x(n-2)} dx + \int_{-\infty}^\infty  \frac{x^n}{\sinh(x)^n} e^{-x(n-2)} dx \right).
\end{align*}
As $\frac{x}{\sinh(x)}$ is an even function, the value of the first integral is unchanged if we change the $e^{x(n-2)}$ in the integrand to $e^{-x(n-2)}$. Thus,
\begin{align*}
    \lim_{t \to 0^+} t^n G(t) &= \frac{n-1}{2^{n-2}n!} \left( \int_{-\infty}^\infty  \frac{x^n}{\sinh(x)^n} e^{-x(n-2)} dx + \int_{-\infty}^\infty  \frac{x^n}{\sinh(x)^n} e^{-x(n-2)} dx \right) \\
    &=\frac{n-1}{2^{n-1}n!} \int_{-\infty}^\infty  \left( \frac{x}{\sinh(x)} \right)^n e^{-x(n-2)} dx\\
    &= \frac{2\pi^n}{(n-1)!}\cdot \frac{n-1}{ (2\pi)^n n} \int_{-\infty}^\infty  \left( \frac{x}{\sinh(x)} \right)^n e^{-x(n-2)} dx \\
    &= \volSn \frac{n-1}{ (2\pi)^n n}\int_{-\infty}^\infty  \left( \frac{x}{\sinh(x)} \right)^n e^{-x(n-2)} dx. 
\end{align*}
Therefore, by the Tauberian Theorem due to Karamata, we obtain the following limit 
    \[\lim_{\la \to \infty} \frac{N(\la)}{\la^n}  =  \volSn \frac{(n-1)}{n (2\pi)^n \Gamma(n+1) }  \int_{-\infty}^\infty  \left( \frac{x}{\sinh(x)} \right)^n e^{-x(n-2)} dx,\]
and complete the proof of Theorem \ref{main}.

\section{The Formula for Functions versus the Formula for Forms}\label{relating}

In \cite{StanTar}, Stanton and Tartakoff prove the following formula, reminiscent of Weyl's law, for the Kohn Laplacian on CR manifolds of hypersurface type.

\begin{thm}[{\cite[Theorem 6.1]{StanTar}}]
    Let $M$ be a CR submanifold of $\mathbb{C}^n$, $n \geq 3$. 
    Let $N(\la)$ be the eigenvalue counting function of $\square_b$ on $M$ acting on $(p,q)$ forms, $0 \leq p < n$, $0 < q < n-1$. Then we have the asymptotic equivalance
    \[\lim_{\la \to \infty} \frac{N(\la)}{\la^n} = c_n \textup{vol}(M) \la^n, \]
    where
    \[c_n = \binom{n-1}{p} \binom{n-1}{q} \frac{1}{(2\pi)^{n} \Gamma(n+1)} \int_{-\infty}^\infty \left(\frac{\tau}{\sinh \tau}\right)^{n-1} e^{-(n-1-2q)\tau} \,d\tau, \]
    and the Kohn Laplacian and the volume of $M$ are defined with respect to a Levi metric.
\end{thm}
For spheres embedded in $\C^n$, the induced metric is a Levi metric (see Definition 1.5 and the following remark in \cite{StanTar}). However, as stated, this only applies to $(p,q)$ forms with $q\geq 1$. We analyze how this expression relates to Theorem \ref{main}. Towards this end, we define the function
\[f(q) = \binom{n-1}{q} \frac{\textup{vol}(\mathbb{S}^{2n-1})}{(2\pi)^{n} \Gamma(n+1)} \int_{-\infty}^\infty \left(\frac{\tau}{\sinh \tau}\right)^{n-1} e^{-(n-1-2q)\tau} \,d\tau,\]
which for $q=1,\dots,n-2$ is the leading coefficient on $\la^{n}$ for the asymptotic growth of $N(\la)$, the eigenvalue counting function of $\square_b$ on $M$ acting on $(0,q)$ forms. 

The following statement shows that this function is closely related to our formula.
\begin{thm}
    The definition of $f$ given above is convergent for complex $q$ satisfying $0 < \Re(q) < n-1$. Further, $f$ is holomorphic on this strip, and has an analytic continuation to a meromorphic function on the strip
    $-1 < \Re(q) < n-1$ whose only pole is at $q=0$. Finally, the Laurent expansion of $f$ about 0 is 
    \[f(q) = a_n/q^n +  \volSn \frac{(n-1)}{n (2\pi)^n \Gamma(n+1) }  \int_{-\infty}^\infty  \left( \frac{x}{\sinh(x)} \right)^n e^{-x(n-2)} dx,\] where $a_n \neq 0$. 
\end{thm}
In other words, the constant term in the Laurent expansion of $f$ about 0 is the expression from Theorem \ref{main}.

\begin{proof}
For notational convenience let $m = n - 1$. Then we may write $f$ as
\[f(q) = \binom{m}{q} \frac{\textup{vol}(\mathbb{S}^{2n-1
})}{(2\pi)^{n} \Gamma(n+1)} \int_{-\infty}^\infty \left(\frac{\tau}{\sinh \tau}\right)^{m} e^{-(m-2q)\tau} \,d\tau. \] We first prove that the integrand is integrable ($L^1$) whenever $0 < \Re(q) < m$. For this, 
choose $C > 0$ and $\alpha > 0$ so that whenever $|x| \geq C$, $|\sinh(x)| \geq \alpha e^{|x|}$. Therefore for $|x| \geq C$ we have
\[\frac{x}{\sinh(x)} = \left|\frac{x}{\sinh(x)}\right| \leq \frac{|x|}{\alpha e^{|x|}}.\]
Since $x/\sinh(x)$ has a removable singularity at 0, it is continuous and thus bounded on $[-C,C].$ Hence $x^m/\sinh(x)^m$ is bounded on $[-C,C]$ as well, so we can choose $D$ so that $x^m/\sinh(x)^m \leq D$ whenever $|x| \leq C$. Thus, we have the bound
\[\frac{x^m}{\sinh(x)^m} \leq \frac{|x|^m}{\alpha^m e^{m |x|}} \ind_{|x| \geq C} + D \ind_{|x| \leq C} \] for all $x \in \R$.

Therefore, we may write
\begin{align*}
    \int_{-\infty}^\infty \left| \left(\frac{\tau}{\sinh \tau}\right)^{m} e^{-(m-2q)\tau} \right| \,&d\tau =
    \int_{-\infty}^\infty \left(\frac{\tau}{\sinh \tau}\right)^{m} e^{-(m-2\Re(q))\tau} \,d\tau \\
    &\leq \alpha^{-m}\int_{-\infty}^C |\tau|^m e^{-m|\tau| - (m-2\Re(q))\tau} \, d\tau + \alpha^{-m} \int_C^{\infty} \tau^m e^{-m\tau - (m-2\Re(q))\tau} \, d\tau + \int_{-C}^{C} D \, d\tau \\
    &= \alpha^{-m}\int_{C}^\infty \tau^m e^{-m\tau + (m-2\Re(q))\tau} \, d\tau + \alpha^{-m} \int_C^{\infty} \tau^m e^{-m\tau - (m-2\Re(q))\tau} \, d\tau + 2CD \\
    &= \alpha^{-m}\int_C^\infty \tau^m e^{-2 \tau \Re(q)} \, d\tau + \alpha^{-m} \int_C^{\infty} \tau^m e^{-2\tau(m-\Re(q))} \, d\tau + 2CD.
\end{align*}
A repeated integration by parts shows that these integrals are finite if $0 < \Re(q) < m.$ 
To define the binomial coefficient, use the gamma function, i.e.
\[\binom{m}{q} = \frac{m!}{\Gamma(q+1)\Gamma(m-q+1)}, \]
which is defined since $\Gamma$ has no zeros. Thus, $f$ is well defined on its domain of definition.

Now, rewrite $f$ as follows:
\begin{align*}
f(q) &= \binom{m}{q} \frac{\textup{vol}(\mathbb{S}^{2n-1
})}{(2\pi)^{n} \Gamma(n+1)} \int_{-\infty}^\infty \left(\frac{\tau}{\sinh \tau}\right)^{m} e^{-(m-2q)\tau} \,d\tau \\
&= \frac{1}{2} \binom{m}{q} \frac{\textup{vol}(\mathbb{S}^{2n-1
})}{(2\pi)^{n} \Gamma(n+1)} \left( \int_{-\infty}^\infty \left(\frac{\tau}{\sinh \tau}\right)^{m} e^{-(m-2q)\tau} \,d\tau + \int_{-\infty}^\infty\left(\frac{\tau}{\sinh \tau}\right)^{m} e^{-(m-2q)\tau} \,d\tau \right).
\end{align*}
As $\frac{x}{\sinh(x)}$ is an even function, the value of the second integral is unchanged if we change the $e^{-(m-2q)\tau}$ in the integrand to $e^{(m-2q)\tau}$. Therefore,
\begin{align*}
f(q) &= \frac{1}{2}\binom{m}{q} \frac{\textup{vol}(\mathbb{S}^{2n-1
})}{(2\pi)^{n} \Gamma(n+1)} \left( \int_{-\infty}^\infty \left(\frac{\tau}{\sinh \tau}\right)^{m} e^{-(m-2q)\tau} \,d\tau + \int_{-\infty}^{\infty} \left(\frac{\tau}{\sinh \tau}\right)^{m} e^{(m-2q)\tau} \,d\tau \right) \\
&= \binom{m}{q} \frac{\textup{vol}(\mathbb{S}^{2n-1
})}{(2\pi)^{n} \Gamma(n+1)} \int_{-\infty}^\infty \left(\frac{\tau}{\sinh \tau}\right)^{m} \cosh((m-2q)\tau) \,d\tau \\
&= 2 \binom{m}{q} \frac{\textup{vol}(\mathbb{S}^{2n-1
})}{(2\pi)^{n} \Gamma(n+1)} \int_{0}^\infty \left(\frac{\tau}{\sinh \tau}\right)^{m} \cosh((m-2q)\tau) \,d\tau
\end{align*}
where the last step follows since the integrand is even. Now, define the function $g$ as follows:
\[
g(q) = 2 \binom{m}{q} \frac{\textup{vol}(\mathbb{S}^{2n-1})}{(2\pi)^{n} \Gamma(n+1)} \int_{0}^{\infty} \tau^m \left(\frac{\cosh((m-2q)\tau))}{(\sinh \tau)^m} - 2^{m-1} e^{-2q\tau} \right)  \,d\tau.
\]
We claim that $g$ is holomorphic on the strip $-1 < \Re(q) < m.$ Assuming this for now, consider 
\[ f(q) - g(q) = 2^m \binom{m}{q} \frac{\textup{vol}(\mathbb{S}^{2n-1})}{(2\pi)^{n} \Gamma(n+1)} \int_{0}^{\infty}  \tau^m  e^{-2q\tau}  \,d\tau \]
defined for $0 < \Re(q) < m$. This is easy to evaluate explicitly with the substitution $u = 2q\tau:$
\[
\int_{0}^{\infty} 
\tau^m e^{-2q\tau}  \,d\tau = 
 \frac{1}{(2q)^m}\int_{0}^{\infty} u^m  e^{-u}\,du = \frac{\Gamma(m+1)}{(2q)^m}
\]
so
\[ f(q) - g(q) = \binom{m}{q} \frac{\textup{vol}(\mathbb{S}^{2n-1})}{(2\pi)^{n}(m+1)} \frac{1}{q^{m+1}} = \binom{n-1}{q} \frac{\textup{vol}(\mathbb{S}^{2n-1})}{n(2\pi)^{n}} \frac{1}{q^{n}}.\]
This is meromorphic in $q$ on the whole complex plane. Thus, the function $g + (f - g)$ is meromorphic in $q$ for $-1 < \Re(q) < n-1$, and equal to $f$ if $0 < \Re(q) < n-1$. Since $f$ converges on its domain of definition, $f$ is holomorphic in $q$, and $g + (f -g)$ is the desired continuation. To complete the proof, we need to show that $g(0)$ is our formula. We have
\begin{align*}
g(0) &= 2 \frac{\textup{vol}(\mathbb{S}^{2n-1})}{(2\pi)^{n} \Gamma(n+1)} \int_{0}^{\infty} \tau^m \left(\frac{\cosh(m\tau)}{(\sinh \tau)^m} - 2^{m-1} \right)  \,d\tau \\
&= 2 \frac{2\pi^n/\Gamma(n)}{(2\pi)^{n} \Gamma(n+1)} \int_{0}^{\infty} \tau^m \left(\frac{e^{m\tau} + e^{-m\tau}}{2^{m-1}(e^\tau - e^{-\tau})^m} - 2^{m-1} \right)  \,d\tau \\
&= \frac{1}{\Gamma(n)\Gamma(n+1)} \int_{0}^{\infty} \tau^m \left(\frac{e^{m\tau} + e^{-m\tau}}{(e^\tau - e^{-\tau})^m} - 1 \right)  \,d\tau \\
&= \frac{1}{(n-1)! \,\Gamma(n+1)} \int_{0}^{\infty} \tau^m \left(\frac{e^{m\tau}}{(e^\tau - e^{-\tau})^m} + \frac{e^{-m\tau}}{(e^\tau - e^{-\tau})^m} - 1 \right)  \,d\tau \\
&= \frac{1}{(n-1)! \, \Gamma(n+1)} \int_{0}^{\infty} \tau^m \left(\frac{1}{(1 - e^{-2\tau})^m} + \frac{1}{(e^{2\tau} - 1)^m} - 1 \right)  \,d\tau \\
&= \frac{1}{\Gamma(n+1)} \lim_{t \to 0} t^n G(t) \\
&= \lim_{\lambda \to \infty} \frac{N(\lambda)}{\lambda^n}
\end{align*}
where we have used the expression appearing in the discussion after Proposition \ref{Gt}, and Karamata's Theorem. Thus we have our theorem, modulo showing that $g$ is holomorphic for $-1 < \Re(q) < m$. We move on to this now. It clearly suffices to show that the expression
\[
h(q) = \int_{0}^{\infty} \tau^m \left(\frac{\cosh((m-2q)\tau))}{(\sinh \tau)^m} - 2^{m-1} e^{-2q\tau} \right)  \,d\tau
\]
is holomorphic for $-1 < \Re(q) < m$. 

As a stepping stone towards proving that $h(q)$ is holomorphic, we will first prove that for $\beta>0$, 
 \[ \phi(\beta) = \int_0^{\infty} e^{-2\beta \tau} \left(\frac{\tau}{1 - e^{-2\tau}}\right)^m d\tau\] is convergent and continuous.
By differentiating with respect to $\beta$, we see that the integrand is motonically increasing, for each $\tau$, as $\beta \to 0.$ Fix some $\beta_0 > 0$. Now, for $\tau$ near 0, the integrand is bounded since it has a limit at 0. For large $\tau$, it is bounded by a constant times
$e^{-2 \beta \tau} \tau^m$, which is integrable on $[0,\infty).$ To see continuity at $\beta_0$, fix some $\beta_1$ with $0 < \beta_1 < \beta_0$, and note that the integrand at $\beta_1$ dominates the integrand at $\beta$ for all $\beta > \beta_1$. Continuity of $\phi$ at $\beta_0$ follows then by dominated convergence. Thus, we have shown that $\phi$ is convergent and continuous for $\beta>0$.

Next, we compute
\[
    \int_{0}^{\infty} \left| \tau^m \left(\frac{\cosh((m-2q)\tau))}{(\sinh \tau)^m} - 2^{m-1} e^{-2q\tau} \right) \right| \,d\tau 
    = 2^{m-1}\int_{0}^{\infty} \tau^m \left|\frac{e^{(m-2q)\tau} + e^{-(m-2q)\tau}}{(e^\tau - e^{-\tau})^m} - e^{-2q\tau} \right| \,d\tau
    \]
 \[   \leq 2^{m-1}\int_{0}^{\infty} \tau^m \left|\frac{ e^{-(m-2q)\tau}}{(e^\tau - e^{-\tau})^m} \right| \,d\tau + 2^{m-1}\int_{0}^{\infty} \tau^m \left|\frac{e^{(m-2q)\tau} }{(e^\tau - e^{-\tau})^m} - e^{-2q\tau} \right| \,d\tau.
\]
To analyze the first term, we have
\[ \int_{0}^{\infty} \tau^m \left|\frac{ e^{-(m-2q)\tau}}{(e^\tau - e^{-\tau})^m} \right| \,d\tau = \int_{0}^{\infty} \tau^m \frac{ e^{-(m-2\Re(q))\tau}}{(e^\tau - e^{-\tau})^m} \,d\tau = 
\int_{0}^{\infty} \tau^m \frac{ e^{-(2m-2\Re(q))\tau}}{(1 - e^{-2\tau})^m} \,d\tau = \phi(m - \Re(q)).
\]
For the second term, we have
\[\int_{0}^{\infty} \tau^m \left|\frac{e^{(m-2q)\tau} }{(e^\tau - e^{-\tau})^m} - e^{-2q\tau} \right| \,d\tau = \int_{0}^{\infty} \tau^m \left|e^{-2q\tau} \left(\frac{e^{m\tau} }{(e^\tau - e^{-\tau})^m} - 1\right) \right| \,d\tau = \int_{0}^{\infty} \tau^m e^{-2\Re(q)\tau} \left| \frac{1}{(1 - e^{-2\tau})^m} - 1 \right| \,d\tau.
\]
We bound the latter expression in the integrand by
\[\left| \frac{1}{(1 - e^{-2\tau})^m} - 1 \right| = \left| \frac{1 - (1- e^{-2\tau})^m}{(1 - e^{-2\tau})^m} \right| = \left| \frac{1 - \sum_{k=0}^m (-1)^k \binom{m}{k} e^{-2k\tau}}{(1 - e^{-2\tau})^m} \right| = \left| \frac{- \sum_{k=1}^m (-1)^k \binom{m}{k} e^{-2k\tau}}{(1 - e^{-2\tau})^m} \right|\]
\[ \leq \frac{ \sum_{k=1}^m \binom{m}{k} e^{-2k\tau}}{(1 - e^{-2\tau})^m}
= e^{-2\tau} \frac{ \sum_{k=1}^m \binom{m}{k} e^{-2(k-1)\tau}}{(1 - e^{-2\tau})^m} \leq e^{-2\tau} \frac{ \sum_{k=1}^m \binom{m}{k}}{(1 - e^{-2\tau})^m} \leq e^{-2\tau}\frac{ 2^m}{(1 - e^{-2\tau})^m}.\]
Thus for the second term we have
\[\int_{0}^{\infty} \tau^m \left|\frac{e^{(m-2q)\tau} }{(e^\tau - e^{-\tau})^m} - e^{-2q\tau} \right| \,d\tau \leq 
\int_{0}^{\infty} \tau^m e^{-2\Re(q)\tau} e^{-2\tau}\frac{ 2^m}{(1 - e^{-2\tau})^m} = 2^m \phi(\Re(q) + 1).\]
In total we have shown that 
\[\int_{0}^{\infty} \left| \tau^m \left(\frac{\cosh((m-2q)\tau))}{(\sinh \tau)^m} - 2^{m-1} e^{-2q\tau} \right) \right| \,d\tau \leq 2^{m-1} \phi(m - \Re(q)) + 2^{2m-1} \phi(\Re(q) + 1). \]
Using this bound, we now show that $h$ is holomorphic via Morera's Theorem. Fix some triangle $\Delta \subset \{q \in \C: -1 < \Re(q) < m\}.$ Parameterize $\Delta$ by arc length with the piecewise differentiable curve $\gamma(t)$, $a \leq t \leq b$ (so $|\gamma'(t)| = 1 )$ for all $t$). Then 
$$\int_{\Delta} h(q) \,dq = \int_a^b \int_{0}^\infty h(\gamma(t)) \gamma'(t) \,d\tau \,dt.$$
The estimate above, 
\[\int_{0}^{\infty} \left| \tau^m \left(\frac{\cosh((m-2 \gamma(t) )\tau))}{(\sinh \tau)^m} - 2^{m-1} e^{-2 \gamma(t) \tau} \right) \gamma'(t) \right| \,d\tau \leq 2^{m-1} \phi(m - \Re(\gamma(t))) + 2^{2m-1} \phi(\Re(\gamma(t)) + 1),\]
is uniformly bounded in $t$ by the compactness of $\Delta$ and the continuity of the upper bound in $t$, so \[
\int_a^b \int_{0}^\infty \left| \tau^m \left(\frac{\cosh((m-2 \gamma(t) )\tau))}{(\sinh \tau)^m} - 2^{m-1} e^{-2 \gamma(t) \tau} \right) \gamma'(t) \right| \,d\tau \,dt < \infty. \] 
Therefore we may apply Fubini's Theorem to see that 
\[ \int_a^b \int_{0}^\infty  \tau^m \left(\frac{\cosh((m-2 \gamma(t) )\tau))}{(\sinh \tau)^m} - 2^{m-1} e^{-2 \gamma(t) \tau} \right) \gamma'(t)  \,d\tau \,dt = \] 
\[\int_{0}^{\infty} \int_{a}^b  \tau^m \left(\frac{\cosh((m-2 \gamma(t) )\tau))}{(\sinh \tau)^m} - 2^{m-1} e^{-2 \gamma(t) \tau} \right) \gamma'(t) \,dt \,d\tau = \int_{0}^{\infty} 0 \, d\tau = 0\]
by Cauchy's Theorem. This shows (by Morera's Theorem) that 
\[
h(q) = \int_{0}^{\infty} \tau^m \left(\frac{\cosh((m-2q)\tau))}{(\sinh \tau)^m} - 2^{m-1} e^{-2q\tau} \right)  \,d\tau
\]
is a holomorphic function of $q$ for $-1 < \Re(q) < m$, and thus $g(q)$ is holomorphic for $-1 < \Re(q) < m$, completing the proof.
\end{proof}

\section*{Acknowledgements} 
We would like to thank Mohit Bansil and Michael Dabkowski for careful comments on an earlier version of this paper. This research was conducted at the NSF REU Site (DMS-1950102, DMS-1659203) in Mathematical Analysis and Applications at the University of Michigan-Dearborn. We would like to thank the National Science Foundation, National Security Agency, and University of Michigan-Dearborn for their support. 
%\section*{Conflict of Interest} 
%On behalf of all authors, the corresponding author states that there is no conflict of interest. 

\newcommand{\etalchar}[1]{$^{#1}$}

\end{document}